\documentclass[reqno]{amsart}
\usepackage{graphicx, enumerate}
\newtheorem*{thm}{Theorem}

\newtheorem{lem}{Lemma}
\newcommand{\norm}[1]{\left\Vert#1\right\Vert}
\newcommand{\Norm}[1]{\left\Vert#1\right\Vert_{(1)}}
\newcommand{\Abs}[1]{\left\Vert#1\right\Vert_{(2)}}
\newcommand{\abs}[1]{\left\vert#1\right\vert}
\newcommand{\rl}{{\mathbb{R}}}
\newcommand{\cx}{{\mathbb{C}}}
\newcommand{\dbar}{\overline{\partial}}
\newcommand{\dom}{\mathrm{Dom}}
\newcommand{\bd}{\mathsf{b}}
\newcommand{\BOne}{{\mathbb{B}_1}}
\newcommand{\BTwo}{\mathbb{B}_2}

\begin{document}

\title{A Class of Domains with noncompact $\dbar$-Neumann operator}
\author{Debraj Chakrabarti}
\address{TIFR Centre for Applicable Mathematics, Sharadanagara, Chikkabommasandra,  Bengaluru- 560 065, India}
\email{debraj@math.tifrbng.res.in}
\begin{abstract}
The $\overline{\partial}$-Neumann operator (the  inverse of the complex Laplacian) is shown
to be noncompact on certain domains in complex Euclidean space. These domains are either higher-dimensional analogs of
the Hartogs triangle, or have such a generalized Hartogs triangle imbedded appropriately
in them.
\end{abstract}
\maketitle
\section{Introduction}

Let $n_1,{n_2}$ be positive integers, and let $n=n_1+{n_2}$. For a point
$z=(z_1,\dots,z_n)\in\cx^n$, we write $'z$ for the
point $(z_1,\dots,z_{n_1})$ in $\cx^{n_1}$, and also write $z'$ for the point $(z_{{n_1}+1},\dots,z_n)$ in $\cx^{n_2}$,
and denote $z=('z,z')$. Let $\alpha>0$, and consider the bounded pseudoconvex domain
$H$ in $\cx^n$ given by
\begin{equation}\label{eq-h} H= \left\{('z,z')\in\cx^n\colon \Norm{'z}<\left(\Abs{z'}\right)^\alpha<1\right\},\end{equation}
where $\Norm{\cdot}$, $\Abs{\cdot}$ are arbitrary  norms on the complex vector spaces $\cx^{n_1}$, $\cx^{n_2}$ respectively.

We may refer to the domain $H$ of \eqref{eq-h} as a {\em  Hartogs Triangle,} a term usually applied to the case ${n_1}={n_2}=\alpha=1$.
This domain $H$ belongs
to the class of domains for which non-compactness of the $\dbar$-Neumann problem is established in this note.
(See \cite{straube} for the relevant definitions,
as well as  a comprehensive discussion of compactness in the $\dbar$-Neumann problem. Other important texts
dealing with the $\dbar$-Neumann problem include \cite{fk,cs}.)
To define precisely the class of domains we will be considering, note that 
 the boundary $\bd H$ of $H$ contains the   piece
\begin{align} S&=  \left\{('z,z')\in\cx^n\colon \Norm{'z}<1, \Abs{z'}=1\right\}\nonumber\\
&= \BOne \times \bd\BTwo,\label{eq-S}
\end{align}
where $\BOne\subset\cx^{n_1}$ and $\BTwo\subset\cx^{n_2}$ are the unit balls in the norms $\Norm{\cdot}$ and $\Abs{\cdot}$
respectively.
In this note, we prove non-compactness of the $\dbar$-Neumann operator on
a bounded pseudoconvex domain $\Omega$ in $\cx^n$,
which has roughly speaking the following property: it is possible to embed the Hartogs triangle
$H$ of \eqref{eq-h} into $\Omega$ in such a way that the subset $S$  of the boundary of $H$
given in \eqref{eq-S} is mapped
into the boundary of $\Omega$. For example, it suffices to assume that there is a biholomorphic
map from a neighborhood of $\overline{H}$ into $\cx^n$, which maps $H$ into $\Omega$ and $S$
into $\bd\Omega$.

 More precisely, it is sufficient to assume the following hypotheses on $\Omega$:  (i)  there exists a  biholomorphic map
\begin{equation}\label{eq-f} F: {H}\to F(H)\subset \Omega,\end{equation}
which  extends  to  a $\mathcal{C}^\infty$-diffeomorphism of a neighborhood
of $\overline{H}$ onto a neighborhood of  $\overline{F(H)}$ in $\cx^n$.
(ii)  Further, $F$ itself extends biholomorphically to $S$ in such a way that
\begin{equation} \label{eq-FS}
F(S)\subset \bd\Omega.
\end{equation}

 For such $\Omega$, we will prove  the following:

\begin{thm}\label{thm-main}
For $1\leq q \leq {n_1}$ the $\dbar$-Neumann
operator of $\Omega$ acting on $L^2_{0,q}(\Omega)$ is non-compact.
\end{thm}
{\em Acknowledgments:} The author thanks Professors Mei-Chi Shaw and Emil Straube for their
help in the preparation of this article,  the referee for his  helpful suggestions for improvements in it,
and Prof. Giuseppe Zampieri for pointing out  the reference \cite{khanh}.

\section{Some remarks}

We note that for each $\zeta\in \bd \BTwo$, the subset $F(S)\subset b\Omega$ of \eqref{eq-FS} contains
the ${n_1}$-dimensional analytic variety $F(\BOne\times\{\zeta\})$, and this means absence of
the classical compactness-entailing property
$({P}_q)$, for $1\leq q\leq {n_1}$. This strongly suggests the noncompactness of the $\dbar$-Neumann operator
in degrees $(0,q)$ for $1\leq q \leq {n_1}$. On the
other hand, $(P_q)$ is not known to be  necessary for compactness, 
so this observation by itself does not show that the $\dbar$-Neumann operator
is noncompact. When $n_2=1$, so that the boundary $\bd\Omega$ contains $n-1= n_1$ dimensional
complex manifolds, the noncompactness can be deduced from a result of Catlin, according to which
the presence of such $(n-1)$-dimensional  manifolds on the boundary of a weakly pseudoconvex domain in $\cx^n$
implies non-compactness of the $\dbar$-Neumann operator on $(0,n-1)$ forms
(see \cite{straubefu}, where it is assumed that $n=2$, but the boundary is allowed to be Lipschitz.) 
In the analogs of this result for higher-codimensional manifolds in the boundary (see \cite[Theorem~4.21]{straube}) one needs
to assume that the boundary is strictly pseudoconvex in the directions transverse to the complex manifold. This 
is not true in general for the  generalized Hartogs triangle of \eqref{eq-h}, since e.g.,
we can take $\Abs{w} = \max_{1\leq j\leq n_2}\abs{w_2}$.  In general, establishing 
compactness is a tricky business (see \cite[Chapter 4]{straube}.)
Our interest in the rather special domains
$\Omega$ stems from the fact that  we are able to demonstrate noncompactness by an explicit and elementary
counterexample in the spirit of \cite{krantz, ligo} and \cite[Proposition~6.3]{khanh} by exploiting the  symmetry of $\Omega$ inherited from 
the rotational symmetry of $H$.

We also note here that the result can also be stated in the situation when $\Omega$ is a relatively
compact Stein domain in an $n$ complex-dimensional Hermitian manifold. Under the hypothesis of existence
of a map $F$ with the  same properties as above, we can prove the  noncompactness of the $\dbar$-Neumann operator on
$L^2_{p,q}(\Omega)$, where $0\leq p\leq n$ and $1\leq q \leq {n_1}$. The changes required in the proof are
purely formal, and for clarity of exposition we stick with the ambient manifold $\cx^n$.

Let $N$ be a positive integer, for $1\leq j \leq N$, let $n_j$ be a positive integer, and let $n= \sum_{j=1}^N n_j$.
Suppose that we fix a norm $\norm{\cdot}_{(j)}$ on the vector space $\cx^{n_j}$, and denote a point
$z\in\cx^n$ by $(z^{(1)},\dots,z^{(N)})$, where $z^{(j)}\in\cx^{n_j}$.
In principle, it should be possible to modify the proof of the theorem, to  prove the noncompactness of the $\dbar$-Neumann operator
on a domain of the form
\[\widetilde{H}= \left\{z\in\cx^n\colon \norm{z^{(1)}}_{(1)}< \norm{z^{(2)}}_{(2)}^{\alpha_2} < \norm{z^{(3)}}_{(3)}^{\alpha_3} <\dots
<\norm{z^{(N)}}_{(N)}^{\alpha_N}<1
\right\},\]
or more generally, on a domain in which $\widetilde{H}$ is appropriately embedded.  The exposition in the general case
will involve complicated notation arising from the more intricate geometry, and  for clarity, we write the proof in
a simple situation.
\section{Preliminaries}
\label{sec-prelim}
Pulling back the canonical metric of $\cx^n$ via the map $F^{-1}$ (with $F$ as in \eqref{eq-f}),
followed by a partition of unity argument, there is a smooth Hermitian metric $h$ on $\overline{\Omega}$, such that
the map $F$ is an isometry on $H$. We use the metric $h$ to define the pointwise inner product on spaces of forms,
and the  volume form on $\overline{\Omega}$  in the standard way,
and let $L^2_{p,q}(\Omega)$ be the $L^2$-space of $(p,q)$-forms defined by this pointwise inner product and 
volume form.
Thanks to \cite[Theorem 1]{celik}, we know that whether or not the $\dbar$-Neumann operator is compact  on
the bounded domain $\Omega$ is independent of the hermitian metric on $\Omega$,
as long as the metric is smooth on $\overline{\Omega}$. Consequently, it will suffice to show that the $\dbar$-Neumann operator defined
with respect to the metric $h$ acting on $L^2_{0,q}(\Omega)$,  (denoted by $\mathsf{N}_{q}$) is noncompact, for $1\leq q \leq {n_1}$.
Since $\mathsf{N}_{q+1}$ is compact if $\mathsf{N}_q$ is compact for $1\leq q \leq n-1$, (see \cite[Proposition~4.5]{straube})
it follows that it is sufficient
to show the operator $\mathsf{N}_{{n_1}}:L^2_{0,{n_1}}(\Omega)\to L^2_{0,{n_1}}(\Omega)$ is non-compact.

Let $\dbar^*$ denote the Hilbert-space  adjoint of $\dbar$ on $L^2_{0,{n_1}}(\Omega)$ with respect to the metric $h$,
and denote by $\mathfrak{H}$ the space
 \begin{equation}\label{eq-mathfrakh}\dom(\dbar)\cap\dom(\dbar^*)\subset  L^2_{0,{n_1}}(\Omega),\end{equation}
which is  a Hilbert space with norm
\begin{equation}\label{eq-normh}\norm{f}_{\mathfrak{H}}= \left(\norm{\dbar f}^2_{L^2_{0,{n_1}+1}(\Omega)}+\norm{\dbar^* f}^2_{L^2_{0,{n_1}-1}(\Omega)}\right)^\frac{1}{2}.\end{equation}
To prove non-compactness of $\mathsf{N}_{{n_1}}$, it is sufficient to show that the inclusion map of
$\mathfrak{H}$ into $L^2_{0,{n_1}}(\Omega)$ is noncompact (\cite[Proposition~4.2]{straube}.) Before
we proceed to show this we compute certain quantities which will be useful in the proof. We will denote the $(2n_2-1)$-dimensional Hausdorff measure
induced on a hypersurface in $\rl^{2n_2}=\cx^{n_2}$ by $d\sigma$. For $\nu>0$, we consider
the function
\begin{equation}\label{eq-gamma} \gamma(\nu) = \int_{\bd \BTwo} \abs{t_{n_2}}^{2\nu} d\sigma(t),
\end{equation}
where $t=(t_1,\dots,t_{n_2})\in\cx^{n_2}$, and the integral extends over the boundary $\bd\BTwo=\{\Abs{t}=1\}$
of  the unit ball $\BTwo$ of the norm $\Abs{\cdot}$ on $\cx^{n_2}$. We will denote by $dV$ the
Lebesgue measure on Euclidean spaces of arbitrary dimensions (the dimension being known from the context.)
 We first  prove a couple of lemmas:
\begin{lem}\label{oldlemma}For $\beta\geq 0$, we have $\displaystyle{\int_{\BTwo} \abs{w_{n_2}}^{2\nu}\Abs{w}^{2\beta} dV(w)= \frac{\gamma(\nu)}{2(\nu+\beta+{n_2})}}$
\end{lem}
\begin{proof}
Using the co-area formula,
we can rewrite the integral on the left as
\[ \int_{r=0}^1 \left(\int_{\{\Abs{w}=r\}}\abs{w_{n_2}}^{2\nu} r^{2\beta} d\sigma(w) \right)dr,\]
where the inner integral extends over the level set $\{\Abs{w}=r\}\subset \cx^{n_2}$ and is with respect
to the surface measure $\sigma$ induced on this set by the ambient Lebesgue measure. To evaluate
the inner integral, we make the substitution $w=rt$, where $\Abs{t}=1$, i.e., $t$
lies in the boundary $\bd \BTwo$ of the unit ball of the $\Abs{\cdot}$ norm.
Taking advantage of the fact that the level sets of the norm are related by dilations,
we can compute this integral as
\begin{align*}
&\int_{r=0}^1 \left(\int_{\bd \BTwo}r^{2\nu} \abs{t_{n_2}}^{2\nu} r^{2\beta} r^{2{n_2}-1} d\sigma(t)\right) dr\\
&= \left(\int_0^1 r^{2\nu+2\beta+2{n_2}-1}dr\right) \left(\int_{\bd \BTwo} \abs{t_{n_2}}^{2\nu}d\sigma(t)\right)\\
&= \frac{1}{2(\nu+\beta+{n_2})}\gamma(\nu).
\end{align*}\end{proof}
We will need the following fact regarding removable singularities (cf. \cite{hp}):
\begin{lem} \label{newlemma}Let $N\geq 3$, and  let $P$ be a first order linear partial differential operator on a domain $U\subset\rl^N$.
Let $q\in U$, and suppose that $u,v\in L^2(U)$ are such that on $U\setminus\{q\}$ we have in the sense of distributions
 $Pu=v$. Then,  as distributions, $Pu=v$ on $U$.
\end{lem}
\begin{proof} For $\epsilon>0$, let $\chi_\epsilon\in\mathcal{C}^\infty_0(\rl^N)$ be a smooth cutoff such that $0\leq \chi_\epsilon\leq 1$,
$\chi_\epsilon$ is equal to 1 near $q$,  $\chi_\epsilon$ vanishes outside the ball $B(q,\epsilon)$, and we have
$\abs{\nabla\chi_\epsilon}=O(\epsilon^{-1})$.  Let $\phi\in \mathcal{C}^\infty_0(U)$. Writing $\phi=(1-\chi_\epsilon)\phi+\chi_\epsilon\phi$,
and denoting the adjoint of $P$ by $P^*$ we have
\begin{align*}
\abs{\langle Pu-v,\phi\rangle}&= \abs{\langle Pu-v,\chi_\epsilon\phi\rangle}\\
&\leq \abs{\langle u,P^*(\chi_\epsilon \phi)\rangle} +\abs{\langle v,\chi_\epsilon\phi\rangle}\\
& \leq C \left(\epsilon^{-1} +1\right)\sqrt{{\rm Vol}(B(q,\epsilon))}\\
&\to 0 {\text{ as } } \epsilon\to 0^+,
\end{align*}
which proves the lemma.\end{proof}
We note here that although Lemma~\ref{newlemma} is stated for operators acting on functions, the result, 
as well as the proof  (after formal changes) continue to hold for differential operators like $\dbar$  that act on sections 
of vector bundles.  
\section{Counterexample to compactness of $\mathsf{N}_{{n_1}}$}
 We now construct  a  sequence $\{u_\nu\}$ bounded in $\mathfrak{H}$ which has no
 convergent subsequence when viewed as a sequence in $L^2_{0,{n_1}}(\Omega)$.
 Let $\chi$ be a  real-valued non-vanishing smooth function on the interval $[0,1]$
 which vanishes in a neighborhood of  $1$.
 Denote by $(z_1, z_2, \dots, z_n)=('z,z')\in\cx^{n_1}\times \cx^{n_2}$ the coordinates
 on $F(H)\subset\Omega$ given by the map $F^{-1}:F(H)\to \cx^n$.
For each positive integer $\nu$  we consider the $(0,{n_1})$-form
\begin{equation}\label{eq-unudef}u_\nu =\begin{cases}\displaystyle{\sqrt{\frac{\nu}{\gamma(\nu)}}} \cdot\chi\left({\displaystyle{\frac{\Norm{'z}}{\Abs{z'}^\alpha}}}\right)
z_n^\nu \cdot \bigwedge_{j=1}^{{n_1}}d\overline{z_j},& \text{on   } F(H),\\
0, \text{ elsewhere.}\end{cases} \end{equation}
(Here $\alpha$ is as in \eqref{eq-h}, and $\gamma$ is as in \eqref{eq-gamma})
Then  $u_\nu$ has support in $F(H)$ , is bounded on $\overline{\Omega}$, and
is smooth everywhere on $\overline{\Omega}$ except  at $F(0)$. In particular,  each $u_\nu\in L^2_{0,n_1}(\Omega)$.
We claim that the sequence $\{u_\nu\}$ is
bounded in $\mathfrak{H}$ but has no
 convergent subsequence as a sequence in $L^2_{0,{n_1}}(\Omega)$.

Before we begin the proof of the claim, we collect a couple of  simple computations.
Note that a norm $\norm{\cdot}$  on  the  vector space is $\cx^k$ is Lipschitz, and hence differentiable
almost everywhere, and it is easy to see that its gradient is bounded, i.e. there is a constant $K_1>0$, depending only
on the norm $\norm{\cdot}$ such that
\[ \abs{\nabla\left({\norm{z}}\right)}\leq K_1, \text{ for almost every $z\in\cx^k$},\]
where $\abs{\cdot}$ denotes the Euclidean norm in $\cx^k$. Now, a computation shows that
\[ \nabla\left(\chi\left(\frac{\Norm{'z}}{\Abs{z'}^\alpha}\right)\right)=\chi'\left(\frac{\Norm{'z}}{\Abs{z'}^\alpha}\right)
\left(\frac{1}{\Abs{z'}^\alpha}\nabla\left( \Norm{'z}\right) -\alpha \frac{\Norm{'z}}{\Abs{z'}^{\alpha+1}}\nabla\left(\Abs{z'}\right)\right),\]
where the gradient is taken in $\cx^n$.
Therefore there is a constant $K_2>0$ such that for any $z\in H$, we have
\begin{equation}\label{eq-nablachi} \abs{\nabla\left(\chi\left(\frac{\Norm{'z}}{\Abs{z'}^\alpha}\right)\right)}
< \frac{K_2}{\Abs{z'}^\alpha} \text{ a. e.} \end{equation}

We now estimate $\norm{\dbar u_\nu}_{L^2_{0,{n_1}+1}(\Omega)}$. Since  $u_\nu\in L^2_{0,n_1}(\Omega)$,
by Lemma~\ref{newlemma} 
the singularity at $F(0)$ of the form $u_\nu$ can be ignored  while computing $\dbar u_\nu$, provided  
$\dbar u_\nu\in L^2_{0,{n_1}+1}(\Omega)$.
Since $z_n^\nu$ is holomorphic, we have on $F(H)$:
\[ \dbar u_\nu =\displaystyle{{\sqrt{\frac{\nu}{\gamma(\nu)}}} \cdot z_n^\nu \cdot\dbar\left(\chi\left({\displaystyle{\frac{\Norm{'z}}{\Abs{z'}^\alpha}}}\right)\right)
 \wedge\left( \bigwedge_{j=1}^{{n_1}}d\overline{z_j}\right) }.\]
  From now on, $C$ denotes a constant independent of
$\nu$, which may be different at different occurrences of the symbol.  Using \eqref{eq-nablachi}, we have,
 \begin{align}
\norm{\dbar u_\nu}_{L^2_{0,{n_1}+1}(\Omega)}&\leq
\displaystyle{\sqrt{\frac{\nu}{\gamma(\nu)}}}\cdot \norm{\abs{\nabla\left(\chi\left(\frac{\Norm{'z}}{\Abs{z'}^\alpha}\right)\right)}z_n^\nu}_{L^2(F(H))}\label{eq-dbarnu0}\\
&\leq C\displaystyle{\sqrt{\frac{\nu}{\gamma(\nu)}}} \cdot \norm{\frac{z_n^\nu}{\Abs{z'}^{\alpha}}}_{L^2(H)}\label{eq-dbarunu1}.
\end{align}

To compute integrals over the Hartogs Triangle, we will use the new coordinates $(v,w)$ induced by
the  homeomorphism $\Phi$
from $\BOne \times (\BTwo\setminus \{0\})\subset \cx^n$ to $H\subset \cx^{n}$ given by
\begin{equation}\label{eq-Phi} \Phi(v,w) = (\Abs{w}^\alpha v,w),\end{equation}
which is differentiable a.e., and which has Jacobian determinant $\Abs{w}^{2\alpha {n_1}}$ a.e. Recalling
that $dV$ denotes the Lebesgue measure on Euclidean space, we have
\begin{align}
\norm{\frac{z_n^\nu}{\Abs{z'}^{\alpha}}}_{L^2(H)}^2 &= \int_H \frac{\abs{z_n}^{2\nu}}{\Abs{z'}^{2\alpha}} dV(z)\nonumber\\
&=\int_{\BOne \times (\BTwo\setminus \{0\})}\frac{\abs{w_{n_2}}^{2\nu}}{\Abs{w}^{2\alpha}} \Abs{w}^{2\alpha {n_1}}
dV(v,w)\nonumber\\
&={\rm Vol}(\BOne)\int_{\BTwo} \abs{w_{n_2}}^{2\nu}\Abs{w}^{2\alpha ({n_1}-1)} dV(w)\nonumber\\
&={\rm Vol}(\BOne)\frac{\gamma(\nu)}{2(\nu+\alpha({n_1}-1)+{n_2})}, \label{eq-w2nu}
\end{align}
where we have used the lemma proved in Section~\ref{sec-prelim}.
Combining \eqref{eq-dbarunu1} and \eqref{eq-w2nu}  gives the estimate
\[\norm{\dbar u_\nu}_{L^2_{0,{n_1}+1}(\Omega)}\leq C \sqrt{\frac{\nu}{\nu+\alpha({n_1}-1)+{n_2}}},  \]
hence combined with the fact that $u_\nu\in L^2_{0,n_1}(\Omega)$ 
each $u_\nu$ is in the domain $\dom(\dbar)$ of the Hilbert space operator $\dbar$,
and there is a $C$ independent of $\nu$ such that
\begin{equation}\label{eq-dbarunu} \norm{\dbar u_\nu}_{L^2_{0,{n_1}+1}(\Omega)}\leq C. \end{equation}

Let $\vartheta$ denote the formal adjoint of $\dbar$ on $L^2_{0,{n_1}}(\Omega)$ with respect to the
metric $h$ defined above. Then on $F(H)$, in the coordinates $(z_1,\dots, z_n)$, the formal expression for
$\vartheta$ coincides with the usual one in Euclidean space. If
\begin{equation}\label{eq-U} U_\nu(z) =\displaystyle{\sqrt{\frac{\nu}{\gamma(\nu)}}} \cdot\chi\left({\displaystyle{\frac{\Norm{'z}}{\Abs{z'}^\alpha}}}\right)
z_n^\nu, \end{equation}
so that on $F(H)$ we have $u_\nu = U_\nu d\overline{z_1}\wedge\dots\wedge d\overline{z_{n_1}}$, we have
on $F(H)$
 \begin{align*} \vartheta u_\nu & = -\sum_{j=1}^{n_1} \frac{\partial U_\nu}{\partial z_j}
 d\overline{z_1}\wedge \dots \wedge\widehat{d\overline{z_j}}\wedge \dots\wedge d\overline{z_{{n_1}}}\\
 &=-\displaystyle{\sqrt{\frac{\nu}{\gamma(\nu)}}}
 \cdot z_n^\nu\cdot\sum_{j=1}^{n_1} \frac{\partial}{\partial z_j}\left(\chi\left({\displaystyle{\frac{\Norm{'z}}{\Abs{z'}^\alpha}}}\right)  \right)
 d\overline{z_1}\wedge \dots \wedge\widehat{d\overline{z_j}}\wedge \dots\wedge d\overline{z_{{n_1}}},\end{align*}
 where the hat on  $\widehat{d\overline{z_j}}$ denotes the omission of this factor from the wedge product. Therefore,
 using  \eqref{eq-nablachi} we compute
\begin{align*}
\norm{\vartheta u_\nu}_{L^2_{0,{n_1}-1}(\Omega)}&= \norm{\vartheta u_\nu}_{L^2_{0,{n_1}-1}(H)}\\
&\leq C\displaystyle{\sqrt{\frac{\nu}{\gamma(\nu)}}}\norm{\abs{\nabla\left(\chi\left(\frac{\Norm{'z}}{\Abs{z'}^\alpha}\right)\right)}z_n^\nu}_{L^2(F(H))},
\end{align*}
which is the same quantity as in \eqref{eq-dbarnu0}. Therefore, the same arguments as used before for $\dbar u_\nu$ show that
 $\vartheta u_\nu\in L^2_{0,{n_1}-1}(\Omega)$, and indeed is
uniformly bounded in $L^2_{0,{n_1}-1}(\Omega)$ independently of $\nu$.

Further, the form $u_\nu$ vanishes everywhere on the boundary $b\Omega$ except
in the patch $\bd\Omega\cap \overline{F(H)}$, which consists of the  disjoint union of the sets $F(S)$ and
$\{F(0)\}\cap b\Omega$, with $F$ as in 
\eqref{eq-f}. (Note that  $\{F(0)\}\cap b\Omega$ is empty if $F(0)\in\Omega$.)
Near $F(S)$ the boundary is represented in the local coordinates $('z,z')\in\cx^{{n_1}+{n_2}}$ as
$\{('z,z')\in\cx^n\colon \Abs{z'}=1\}$. From the formula \eqref{eq-unudef} we see that the complex-normal component of $u_\nu$
vanishes on $\bd\Omega\cap F(S)$. 
For $\epsilon>0$, let $\chi_\epsilon$ be a cutoff of the type used in the proof of Lemma~\ref{newlemma}:
$0\leq \chi_\epsilon\leq 1$,  $\chi_\epsilon\equiv 1$ near $F(0)$,  $\chi_\epsilon$ vanishes outside $B(F(0),\epsilon)$ and
 $\abs{\nabla\chi_\epsilon}= O(\epsilon^{-1})$. Note from the definition that $u_\nu$ has bounded coefficients.
 Writing $\psi_\epsilon=1-\chi_\epsilon$ we have
\[ \vartheta(\psi_\epsilon u_\nu)= U_\nu (z)\cdot\left(-\sum_{j=1}^{n_1} \frac{\partial \psi_\epsilon}{\partial z_j}
 d\overline{z_1}\wedge \dots \wedge\widehat{d\overline{z_j}}\wedge \dots\wedge d\overline{z_{{n_1}}}\right) +\psi_\epsilon \vartheta u_\nu,\]
where $U_\nu$ is as in \eqref{eq-U} on $F(H)$ and extended as zero elsewhere.
The second term approaches $\vartheta u_\nu$ in 
the $L^2$-topology as $\epsilon\to 0^+$ and the $L^2$-norm of the first term is bounded by 
\[ C\norm{U_\nu}_{L^\infty(\Omega)}\epsilon^{-1} \sqrt{{\rm Vol}(B(F(0),\epsilon))},\] which goes to 0 as
 $\epsilon\to 0^+$. Note that $\psi_\epsilon u_\nu$ is  in $\dom(\dbar^*)$, since it is a smooth $(0,n_1)$-form on $\overline{\Omega}$
 whose complex-normal
 component vanishes along $F(S)\subset\bd\Omega$, and $\psi_\epsilon u_\nu$ itself vanishes elsewhere on the boundary.  Further, 
 $\psi_\epsilon u_\nu\to u_\nu$  in the graph norm of $\vartheta$. It follows that  $u_\nu\in \dom(\dbar^*)$,
 and from the computation above,   $\norm{\dbar^*u_\nu}_{L^2_{0,{n_1}-1}(\Omega)}\leq C$, where $C$ does not depend on $\nu$.
Combining this with \eqref{eq-dbarunu}, it follows that each $u_\nu$ lies in the space $\mathfrak{H}$ of \eqref{eq-mathfrakh}, and we have
\[ \norm{u_\nu}_{\mathfrak{H}} \leq C.\]

%--------------------------------------------------------------------------------

We now compute $\norm{u_\nu}_{L^2_{0,{n_1}}(\Omega)}$.  We again use the change of coordinates
given by \eqref{eq-Phi}:
\begin{align*}
\norm{u_\nu}_{L^2_{0,{n_1}}(\Omega)}^2&=\norm{u_\nu}_{L^2_{0,{n_1}}(F(H))}^2\\
&= \displaystyle{\frac{\nu}{\gamma(\nu)}}\int_H \abs{ \chi\left(\frac{\Norm{'z}}{\Abs{z'}^\alpha}\right) z_n^\nu}^2dV(z)\\
&= \displaystyle{\frac{\nu}{\gamma(\nu)}} \int_{\BOne \times (\BTwo\setminus \{0\})} \chi^2(\Norm{v}) \abs{w_{n_2}}^{2\nu} \Abs{w}^{2\alpha {n_1}} dV(v,w)\\
&= \displaystyle{\frac{\nu}{\gamma(\nu)}} \int_{\BOne} \chi^2(\Norm{v}) dV(v) \times \int_{\BTwo} \abs{w_{n_2}}^{2\nu} \Abs{w}^{2\alpha {n_1}} dV(w)\\
&= C \displaystyle{\frac{\nu}{\gamma(\nu)}}\times \frac{\gamma(\nu)}{2(\nu+\alpha {n_1}+{n_2})}\\
&=C\displaystyle{\frac{\nu}{\nu+\alpha {n_1}+{n_2}}},
\end{align*}
where in the line before the last, we have made use of the Lemma~\ref{oldlemma}. Therefore
we see that  there is a constant $\lambda>0$ independent of $\nu$ such that
$\norm{u_\nu}_{L^2_{0,{n_1}}(\Omega)}\geq \lambda$.

Now let $\mu$ and $\nu$ be distinct positive integers. We have, using the same change of variables \eqref{eq-Phi}:
\begin{align*}
\left(u_\mu, u_\nu\right)_{L^2_{0,{n_1}}(\Omega)}&=\left(u_\mu, u_\nu\right)_{L^2_{0,{n_1}}(H)}\\
&=\displaystyle{\frac{\nu}{\gamma(\nu)}}\int_H
 \chi^2\left(\frac{\Norm{'z}}{\Abs{z'}^\alpha}\right)z_n^\mu\overline{z_n}^\nu dV(z)\\
&=\displaystyle{\frac{\nu}{\gamma(\nu)}} \int_{\BOne \times (\BTwo\setminus \{0\})} \chi^2\left(\Norm{v}\right) w_{n_2}^\mu \overline{w_{n_2}}^\nu \Abs{w}^{2\alpha {n_1}} dV(v,w)\\
&= \displaystyle{\frac{\nu}{\gamma(\nu)}} \int_{\BOne} \chi^2(\Norm{v}) dV(v) \times \int_{\BTwo} w_{n_2}^\mu \overline{w_{n_2}}^\nu \Abs{w}^{2\alpha {n_1}} dV(w).
\end{align*}
Denote the  integral over $\BTwo$ in this product by $\mathsf{K}$. We claim that $\mathsf{K}=0$.
Assuming this for a moment it follows that  $\left(u_\mu, u_\nu\right)_{L^2_{0,{n_1}}(\Omega)}=0$,
and hence $\norm{u_\mu-u_\nu}_{L^2_{0,{n_1}}(\Omega)}\geq \sqrt{2}\lambda$ for each pair of distinct indices $\mu$ and $\nu$, and
hence $\{u_\nu\}$ has no convergent subsequence in $L^2_{0,{n_1}}(\Omega)$. The proof of the theorem is complete,
provided we can show that $\mathsf{K}=0$.

Let $\theta$ be a real number which is not a rational multiple of $\pi$,
and in the integral representing $\mathsf{K}$, we make a change of variables $w=e^{i\theta} t$, where $t\in\cx^{n_2}$.
Then, since  this transformation maps $\BTwo$ to itself and its Jacobian  is identically 1, we have
\begin{align*}
\mathsf{K} &= \int_{\BTwo} w_{n_2}^\mu \overline{w_{n_2}}^\nu \Abs{w}^{2\alpha {n_1}} dV(w)\\
&= \int_{\BTwo} e^{i(\mu-\nu)\theta}t_{n_2}^\mu \overline{t_{n_2}}^\nu \Abs{t}^{2\alpha {n_1}} dV(t)\\
&= e^{i(\mu-\nu)\theta} \mathsf{K},
\end{align*}
so that $\mathsf{K}=0$. This completes the proof.
\section{Concluding remarks}

The noncompactness of the $\dbar$-Neumann operator implies that there is a  point  in the spectrum of
the Complex Laplacian $\Box_q$ acting on $L^2_{0,q}(\Omega)$ which is not an eigenvalue of finite multiplicity. It is of
interest to determine the nature of this essential spectrum. In the case of the polydisc  or product
domains, the spectrum can be determined explicitly (see \cite{fu,chak}), and while the spectrum
of a polydisc  consists of eigenvalues only, there are infinitely
many eigenvalues in the essential spectrum, each of infinite multiplicity (this is true, in particular, for the 
smallest nonzero eigenvalue.) It would be of interest to know
whether the same is true for the domain $H$, for example in the classical situation $n_1=n_2=\alpha=1$.

%------------

\end{document}